\documentclass[11pt]{amsart}

\usepackage{amssymb,amsmath,enumerate,epsfig}%,epsf}%t1enc
\usepackage{amsfonts,color}
\usepackage{a4,latexsym}
\usepackage{parskip}
\usepackage[all]{xy}
\usepackage{hyperref}
\hypersetup{pdftitle=dyn}

\newcommand{\C} {\mathbb{C}}
\newcommand{\Q} {\mathbb{Q}}
\newcommand{\N}  {\mathbb{N}}

\newcommand{\Z}{\mathbb{Z}}

\newcommand{\Num}{\mathop{\rm Num}}

\newcommand{\Div}{\mathop{\rm Div}}

\newtheorem{Theorem}{Theorem}[section]
\newtheorem{Lemma}[Theorem]{Lemma}

\newtheorem{Corollary}[Theorem]{Corollary}

\theoremstyle{remark}

\newtheorem{Remark}[Theorem]{Remark}

\theoremstyle{definition}

\newtheorem{Definition}[Theorem]{Definition}

\newtheorem{Convention}[Theorem]{Convention}

\begin{document}

\title{Divisibilities among nodal curves}

\author{Matthias Sch\"utt}
\address{Institut f\"ur Algebraische Geometrie, Leibniz Universit\"at
  Hannover, Welfengarten 1, 30167 Hannover, Germany}

    \address{Riemann Center for Geometry and Physics, Leibniz Universit\"at
  Hannover, Appelstrasse 2, 30167 Hannover, Germany}

\email{schuett@math.uni-hannover.de}

%\subjclass[2010]{14J28; 14J27}
%%
%\keywords{K3 surface, wild automorphism, Lefschetz fixed point formula}
%%
%%

%
\date{May 31, 2017}

\begin{abstract}
We prove that there are no effective or anti-effective  classes
of square $-1$ or $-2$
 arising from nodal curves
on smooth algebraic surfaces by way of divisibility.
This general fact has interesting applications to Enriques and K3 surfaces.
The proof relies on specific properties of root lattices and their duals.
\end{abstract}
%
%
% \begin{abstract}
%This paper concerns K3 surfaces with automorphisms of order 11 in arbitrary characteristic.
%Specifically we study the wild case and prove that a general such surface in characteristic 11
%has Picard number 2.
%We also construct K3 surfaces with an automorphism of order 11 in every characteristic,
%and supersingular K3 surfaces whenever possible.
% \end{abstract}
%% 
 \maketitle

 \section{Introduction}
 \label{s:intro}
 
 Nodal curves are among the most intriguing objects on algebraic surfaces.
The terminology refers to smooth rational curves $C$ with self-intersection $C^2=-2$.
By the adjunction formula, nodal curves can be contracted 
to rational double point singularities
without affecting the dualizing sheaf;
this offers one explanation for our interest in configurations of nodal curves.
%(see...)

Consider the  classes of the nodal curves, or more generally $(-2)$-curves, in the N\'eron--Severi lattice 
\[
\Num(S) = \Div(S)/\equiv
\]
of divisors modulo numerical equivalence
(this equals the N\'eron--Severi group modulo the torsion).
Then divisibilities can lead to the existence of 
certain coverings of $S$ which are often interesting in their own right.
Classically this problem has been studied for K3 surfaces (which we shall come back to momentarily
-- think of Kummer surfaces, for instance).
Here we will develop a general result which we hope to be of independent relevance:

\begin{Theorem}
\label{thm}
Let $R\subset \Num(S)$ be a root lattice generated by $(-2)$-curves on a smooth algebraic surface $S$.
Denote the primitive closure by
\[
R'=(R\otimes\Q)\cap \Num(S) \;\;\; 
\text{ and let } \;\;\;
D\in R'\setminus R.
\]
If $D^2=-2$ or $-1$,
then $D$ is neither effective nor anti-effective.
\end{Theorem}

The proof of Theorem \ref{thm} largely builds on basic properties of root lattices and their duals,
especially in relation with reflections,
which we will discuss in the subsequent sections before the proof is completed in Section \ref{s:pf}.
Here we would like to point out two interesting applications to K3 surfaces and Enriques surfaces 
-- where $(-2)$-curves are automatically smooth rational (i.e. nodal) by adjunction.

\begin{Corollary}
Let $S$ be a K3 surface and  $R\subset \Num(S)$ be a root lattice generated by nodal curves on $S$.
Its primitive closure $R'$
contains no vectors outside $R$ with square $-2$:
\[
D\in R', \; D^2>-4 \; \Longrightarrow \; D\in R.
\]
\end{Corollary}

The corollary follows immediately from Theorem \ref{thm} as a consequence of Riemann--Roch
(and the evenness of $\Num(S)$).
It has received some attention before in special cases,
for instance for $R$ decomposing into orthogonal copies of the root lattice $A_2$
(see \cite{b}, \cite{br}, \cite{RS}), or for Kummer surfaces.

Theorem \ref{thm} also has a big impact on possible configurations of nodal curves
on Enriques surfaces: here the configuration often forces divisibilities
(since $\Num(S)$ is unimodular), so Theorem \ref{thm} provides some severe restrictions
-- without appeal to any K3 surfaces, so that this  holds for all Enriques surfaces, even in characteristic two.
For details, the reader is referred to \cite{S-Q-hom}, \cite{S-Q_l}, \cite{S-Q_2}.

\begin{Convention}
Throughout this paper, we take root lattices to be negative-definite
(in agreement with the geometric picture from Theorem \ref{thm}).

\end{Convention}

\section{Basics on Root lattices}

We start by reviewing basic properties of root lattices.
Standard references include \cite{bourbaki-lie}, \cite{CS}.

Any root lattice $R$ admits an orthogonal decomposition into irreducible root lattices $R_i$ of ADE-type,
unique up to order:
\begin{eqnarray}
\label{eq:oplus}
R = \bigoplus_{j=1}^m R_j.
\end{eqnarray}
Given an embedding into some integral lattice $L$,
\[
R \hookrightarrow L
\]
(where we will later take $L=\Num(S)$),
the primitive closure
\[
R' = (R\otimes \Q)\cap L
\]
naturally sits inside the dual of $R$:
\[
R' \hookrightarrow R^\vee = \bigoplus_i R_i^\vee.
\]
For ease of notation, assume for the moment that $R$ is irreducible of rank $n$.
We can thus interpret the elements in $R'\setminus R$ as elements of $R^\vee\setminus R$,
classified modulo $R$ by the non-zero elements of the discriminant group
\[
A_R = R^\vee/R.
\]
This is a finite abelian group whose shape depends on $R$ as detailed in the Table \ref{T1}.
The values  modulo $2\Z$ which the quadratic form assumes on $R^\vee$ only
depend on the representatives of $A_R$. These can easily be worked out, 
for instance in terms of certain dual vectors for the standard basis of $R$
corresponding to the associated Dynkin diagram:

\begin{figure}[ht!]
\setlength{\unitlength}{.6mm}
\begin{picture}(80,15)(-5,20)
%\put(-15,31){$A_n$}
\multiput(3,32)(20,0){5}{\circle*{1.5}}
\put(3,32){\line(1,0){80}}
\put(2,25){$a_1$}
\put(22,25){$a_2$}
\put(50,25){$\hdots$}
\put(82,25){$a_n$}
\put(100,30){$(A_n)$}
%
%\put(-15,7){$D_r:$}
%\multiput(3,8)(20,0){5}{\circle*{3}}
%\put(3,8){\line(1,0){80}}
%\put(83,8){\line(2,1){17}}
%\put(83,8){\line(2,-1){17}}
%\put(100,16){\circle*{3}}
%\put(100, 0){\circle*{3}}
%\put(2,1){$\alpha_1$}
%\put(22,1){$\alpha_2$}
%\put(80,1){$\alpha_{r-2}$}
%\put(105,16){$\alpha_{r-1}$}
%\put(105, 0){$\alpha_{r}$}
%
%\put(-15,-33){$E_r:$}
%\multiput(3,-32)(20,0){6}{\circle*{3}}
%\put(3,-32){\line(1,0){100}}
%\put(2,-39){$\alpha_1$}
%\put(22,-39){$\alpha_2$}
%\put(42,-39){$\alpha_3$}
%\put(62,-39){$\alpha_4$}
%\put(43,-32){\line(0,1){20}}
%\put(43,-12){\circle*{3}}
%\put(47,-13){$\alpha_r$}
%\put(102,-39){$\alpha_{r-1}$}

\end{picture}
%\caption{Dynkin diagram of type $A_n$}  
%\label{DynkinA}
\end{figure}

\begin{figure}[ht!]
\setlength{\unitlength}{.6mm}
\begin{picture}(100,30)(-7,-5)
%\put(-15,31){$A_r:$}
%\multiput(3,32)(20,0){5}{\circle*{3}}
%\put(3,32){\line(1,0){80}}
%\put(2,25){$\alpha_1$}
%\put(22,25){$\alpha_2$}
%\put(82,25){$\alpha_r$}
%
%\put(-15,7){$D_k:$}
\multiput(3,8)(20,0){5}{\circle*{1.5}}
\put(3,8){\line(1,0){80}}
\put(83,8){\line(2,1){17}}
\put(83,8){\line(2,-1){17}}
\put(100,16){\circle*{1.5}}
\put(100, 0){\circle*{1.5}}
\put(2,1){$d_1$}
\put(22,1){$d_2$}
\put(78,1){$d_{n-2}$}
\put(103,16){$d_{n-1}$}
\put(103, 0){$d_{n}$}
\put(120,8){$(D_n)$}

%\put(-15,-33){$E_r:$}
%\multiput(3,-32)(20,0){6}{\circle*{3}}
%\put(3,-32){\line(1,0){100}}
%\put(2,-39){$\alpha_1$}
%\put(22,-39){$\alpha_2$}
%\put(42,-39){$\alpha_3$}
%\put(62,-39){$\alpha_4$}
%\put(43,-32){\line(0,1){20}}
%\put(43,-12){\circle*{3}}
%\put(47,-13){$\alpha_r$}
%\put(102,-39){$\alpha_{r-1}$}

\end{picture}
%\caption{Dynkin diagrams of type $D_k$}  
%\label{DynkinAD}
\end{figure}

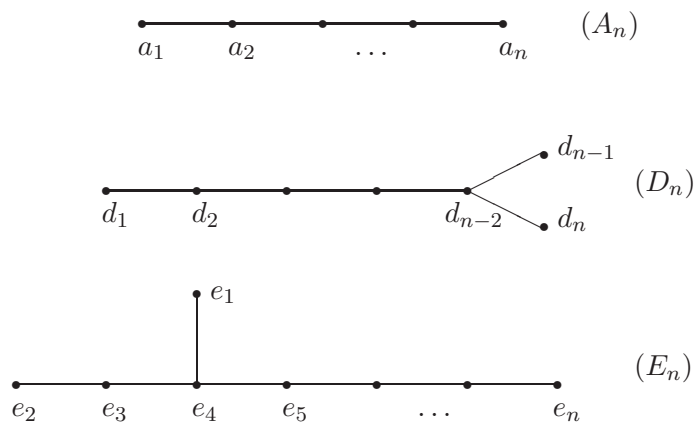
\begin{figure}[ht!]
\setlength{\unitlength}{.6mm}
\begin{picture}(120,30)(3,-40)
%\put(-15,31){$A_r:$}
%\multiput(3,32)(20,0){5}{\circle*{3}}
%\put(3,32){\line(1,0){80}}
%\put(2,25){$\alpha_1$}
%\put(22,25){$\alpha_2$}
%\put(82,25){$\alpha_r$}
%%
%\put(-15,7){$D_r:$}
%\multiput(3,8)(20,0){5}{\circle*{3}}
%\put(3,8){\line(1,0){80}}
%\put(83,8){\line(2,1){17}}
%\put(83,8){\line(2,-1){17}}
%\put(100,16){\circle*{3}}
%\put(100, 0){\circle*{3}}
%\put(2,1){$\alpha_1$}
%\put(22,1){$\alpha_2$}
%\put(80,1){$\alpha_{r-2}$}
%\put(105,16){$\alpha_{r-1}$}
%\put(105, 0){$\alpha_{r}$}
%
%\put(-15,-33){$E_r:$}
\multiput(3,-32)(20,0){7}{\circle*{1.5}}
\put(3,-32){\line(1,0){120}}
\put(2,-39){$e_2$}
\put(22,-39){$e_3$}
\put(42,-39){$e_4$}
\put(62,-39){$e_5$}
\put(43,-32){\line(0,1){20}}
\put(43,-12){\circle*{1.5}}
\put(46,-13){$e_1$}
\put(92,-39){$\hdots$}
\put(122,-39){$e_n$} 
\put(140,-30){$(E_n)$}

\end{picture}
\caption{Dynkin diagrams of type $A_r, D_r, E_r$}  
\label{DynkinAD}
\end{figure}

$$
\begin{array}{c|c|c|c|c|c|c}
R & A_n & D_n (n\geq 4 \text{ even}) & D_n (n>4 \text{ odd}) & E_6 & E_7 & E_8\\
\hline
A_R & \Z/(n+1)\Z & \Z/2\Z\times \Z/2\Z &\Z/4\Z & \Z/3\Z & \Z/2\Z & 0\\
\text{generators} & a_1^\vee &
d_1^\vee, d_n^\vee &
d_n^\vee & e_6^\vee & e_7^\vee & 0
\\
\text{squares} & 
-\frac{n}{n+1} &
-1, -\frac n4 &
-\frac n4 &
-\frac 43 &
-\frac 32 & 0
\\
\end{array}
$$
\begin{table}[ht!]
\caption{Discriminant group data}
\label{T1}
\end{table}
%
%\begin{Remark}
%It will play an important role later on
%that the generators from the table,
%when expressed in terms of the given $\Q$-basis of $R^\vee$,
%all have $\Q$-rational entries negative as one can easily check.
%\end{Remark}
%

Since $E_8$ is unimodular, in particular $E_8 = E_8^\vee$,
we will not have to consider $E_8$ for our purposes
and thus omit this lattice for the rest of this paper without further mention.

\section{Dual vectors}

We continue by elaborating on some important properties of the dual vectors of $R$.
Here we shall mostly be concerned with very specific ones
which relate to the appearance of $R$ in the singular fibers of elliptic fibrations
(as classified by Kodaira over $\C$)
minus any simple component.

\begin{Definition}
We call a vertex $v$ of $R$ simple if the corresponding fiber component in 
the resulting singular fiber is simple.
\end{Definition}

Equivalently $v$ has multiplicity $1$ in the primitive isotropic divisor in the 
extended Dynkin diagram $\tilde R$.
Concretely, all vertices of root lattices of type $A_n$ are simple,
but only 
\[
d_1, d_{n-1}, d_n \; \text{ for } D_n \;\;\; \text{ and } \;\;\; e_2, e_6 \; \text{ for } E_6, \;\;
e_7 \; \text{ for } E_7.
\] 
The meaning of simple vertices  is illustrated by the following interrelated classical facts:

\begin{Lemma}
The non-zero elements of $A_R$ are exactly represented by the dual vectors $v^\vee$
for the simple vertices $v\in R$.
\end{Lemma}

Since $R^\vee\subset R\otimes\Q$, we can express each  vector $v\in R^\vee$
in terms of the standard basis $v_i$ as a sum
\begin{eqnarray}
\label{eq:v}
v=\sum_{i=1}^n \alpha_i v_i, \;\;\; \alpha_i\in\Q.
\end{eqnarray}

\begin{Lemma}
\label{lem:a_i}
If $v\in R$ is a simple vertex, then the dual vector $v^\vee$ has all $\alpha_i<0$.
\end{Lemma}

Moreover the simple vertices feature an intriguing relation to the roots of $R$:

\begin{Lemma}
Let $x\in R$ be a root,
expressed in terms of the standard basis $v_i$ as $x=\sum_{i=1}^n \alpha_i v_i$.
Then $\alpha_i\in\{-1,0,1\}$ for all simple vertices $v_i$.
\end{Lemma}

We note an important consequence
which will be instrumental for our considerations to follow.

\begin{Corollary}
\label{cor:v.x}
If $v$ is a simple vertex and $x$ a root in $R$,
then $v^\vee.x\in\{-1,0,1\}$.
\end{Corollary}

\begin{Remark}
\label{rem:no}
The statement of Corollary \ref{cor:v.x} does no longer hold true
if $v$ is not a simple vertex.
Indeed, then there is some root $x\in R$ such that $v^\vee.x\geq 2$.
\end{Remark}

The statement of Corollary \ref{cor:v.x} does, however, 
generalize to certain other elements in $R^\vee$ as we shall see and utilize crucially in Lemma \ref{lem:w.x}.

\section{Small vectors}

Recall that in the end we are determined to study vectors of square $-1$ or $-2$ in $R'\subset R^\vee$.
Necessarily these are composed of small (or zero) vectors from the dual lattices $R_i^\vee$
of the orthogonal summands which we shall thus study in a little more detail here.

\begin{Convention}
Following the standard terminology inspired from the positive-definite case,
we let 'small' refer to the absolute value.
\end{Convention}

\subsection{Smallest vectors}
\label{ss:1st}

If $R$ denotes an irreducible root lattice as before,
then
the value of the smallest non-zero vectors in $R^\vee$ is exactly given in the Table \ref{T1}.
-- except for $D_n$ for odd $n>4$ where $(d_1^\vee)^2=-1$ attains the minimum.
%A priori, a smallest vectors $v_i$ may always contribute to $v$,
%since $v_i^2>-1$ (although expressing the difference $v^2-v_i^2$ by other  vectors from the $R_i^\vee$ gives severe restrictions
%in case $R_i=A_n$, notably because either $v_i^2=0$ or $v_i^2\leq -\frac 12$).

It is well known that the smallest vectors form exactly one, two or three orbits under the action of the Weyl group $W(R)$,
and that each orbit is generated by the dual vector of some simple vertex.
In detail, we have:

$$
\begin{array}{c||c|c|c|c|c|c}
R & A_1 & A_n \, (n>1) & D_4 & D_n \, (n>4) & E_6 & E_7\\
\hline
\# \text{ orbits} & 1 & 2 & 3 & 1 & 2 & 1\\
\text{generators} & 
a_1^\vee &
a_1^\vee, a_n^\vee &
d_1^\vee, d_2^\vee, d_3^\vee &
d_1^\vee &
e_2^\vee, e_6^\vee &
e_7^\vee
\end{array}
$$

%
%generated by 
%\[
%a_1^\vee (A_1), 
%\pm a_1^\vee (A_n, n>1),
%d_1^\vee, d_2^\vee, d_3^\vee (D_4),
%d_1^\vee (D_n, n>4),
%\pm e_6^\vee (E_6), 
%e_7^\vee (E_7).
%\]
%size
%
%

\subsection{2nd smallest vectors}
\label{ss:2nd}

For the second smallest non-zero vectors in the dual $R^\vee$ of an irreducible root lattice,
we did not find a proper reference, 
but all claims in this paragraph (and in \ref{ss:3rd}, \ref{ss:4th})
can be verified directly 
(for instance, utilizing the standard model of $A_n$ as
trace zero hypersurface inside euclidean  $\Z^{n+1}$, up to sign,
with Weyl group $W(A_n)\cong S_{n+1}$ acting by permutation of coordinates).
%\cap \{\sum_{i=1}^{n+1} \alpha_i=0\}$)
%so we verified all the claims in this paragraph (and in \ref{ss:3rd}, \ref{ss:4th})
%by hand by computing the number of vectors with given square (for instance, with the help of some computer algebra system
%such as MAGMA)
%and comparing with the $W(R)$-orbits of the dual vectors (of the given square) of the simple vertices in $R$.

The second smallest non-zero vectors in $R^\vee$ have square
\begin{itemize}
\item
$(a_2^\vee)^2 = -\frac {2(n-1)}{n+1}$ for $A_n \, (n>2)$,
\item
$(d_n^\vee)^2 = -\frac n4$ for $D_n \, (4< n \leq 8)$, and
\item
 $v^2=-2$, attained by any root  $v\in R$, but by no element in $R^\vee\setminus R$,
for all other irreducible root lattices $R$.
\end{itemize}
Obviously, the third case cannot contribute to our problem,
so we simply analyse the first two settings.
%(and the first only in a certain range which will be sufficient for our purposes,
%cf.~Lemma \ref{lem:range}).
$$
\begin{array}{c||c|c|c|c}
R & A_3 & A_n \, (n>3) & D_n \, (4<n\leq 7) & D_8\\
\hline
\# \text{ orbits of 2nd smallest vectors} & 1 & 2 & 2 & 3\\
\text{generators} & 
a_2^\vee &
a_2^\vee, a_{n-1}^\vee &
d_{n-1}^\vee, d_n^\vee & 
d_{n-1}^\vee, d_n^\vee, \text{ any root}
\end{array}
$$

\subsection{3rd smallest vectors}
\label{ss:3rd}

With the third smallest vectors,
the problem becomes even easier because there are only a finite number of cases left
with value $\geq -2$ attained by some vector in $R^\vee\setminus R$.
Indeed, all $D_n$ and $E_n$ lattices are excluded anyway,
same for $A_1,\hdots, A_4$,
and for $A_n \, (n>4)$, 
the 3rd smallest value in $A_n^\vee\setminus A_n$ is attained by
\[
(a_3^\vee)^2 = -\frac {3(n-2)}{n+1}
\]
which exceeds $-2$ already starting from $n=9$.
For the remaining root lattices, we compute:
$$
\begin{array}{c||c|c|c}
R & A_5 & A_n \, (n=6,7) & A_8\\
\hline
\# \text{ orbits of 3rd smallest vectors} & 1 & 2 & 3\\
\text{generators} & 
a_3^\vee &
a_3^\vee, a_{n-2}^\vee &
a_{3}^\vee, a_6^\vee, \text{ any root}
\end{array}
$$

\subsection{4th smallest vectors}
\label{ss:4th}

Along the same lines, the case of the 4th smallest non-zero vector
boils down to the root lattice $A_7$ right away
(as far as our problem is concerned),
with value $-2$ and two orbits generated by $a_4^\vee$ and any root of $A_7$.

For the record, we point out that the fifth smallest non-zero vectors $w$ in $R^\vee$
(and beyond) always fulfill $w^2<-2$.

\subsection{Applications to intersection numbers}

We note an innocent, but very useful consequence of our above findings
for intersection numbers with roots in $R$:

\begin{Lemma}
\label{lem:w.x}
Let $w\in R^\vee\setminus R$ be a small vector in the range of \ref{ss:1st} -- \ref{ss:4th}.
Then for any root $x\in R$ we have
\[
w.x\in\{-1,0,1\}.
\]
\end{Lemma}

\begin{proof}
By our previous considerations, there is some element $g\in W(R)$,
with action extended to $R^\vee\subset R\otimes\Q$,
 mapping $w$ to 
the dual vector $v^\vee$ of some simple vertex $v\in R$.
Hence
\[
w.x = g(w).g(x) = v^\vee.g(x) \in \{-1,0,1\}
\]
by Corollary \ref{cor:v.x} since $g(x)$ is, of course, again a root in $R$.
\end{proof}

Applied to the standard basis $v_1,\hdots, v_n$ of $R$,
this has an interesting consequence which will become important in the next section:

\begin{Lemma}
\label{lem:2cases}
Let $w\in R^\vee\setminus R$ be a small non-zero vector in the range of \ref{ss:1st} -- \ref{ss:4th}.
Then 
\begin{itemize}
\item
either there is some $i\in\{1,\hdots,n\}$ such that $w.v_i=-1$ 
\item
or there is some simple vertex $v\in R$ such that $w=v^\vee$.
\end{itemize}
\end{Lemma}

\begin{proof}
By Lemma \ref{lem:w.x}, we have $w.v_i\in\{-1,0,1\}$ for all $i=1,\hdots,n$.
Assume that $w.v_i\geq 0$ for all $i$.
Since $R$ is non-degenerate and $w\neq 0$,
there is some $i$ such that $w.v_i=1$.
Assume that there is $j\neq i$ with $w.v_j=1$ as well.
Then consider the root $x=v_1+\hdots+v_n$.
By assumption, $w.x\geq 2$,
contradicting Lemma \ref{lem:w.x}.
Hence $w.v_j=0$ for all $j\neq i$,
and thus $w=v_i^\vee$.
Here $v_i$ is a simple vertex since by Lemma \ref{lem:w.x},
$w.x\in\{-1,0,1\}$ for any root in $R$ which does not
hold for any other vertex by Remark \ref{rem:no}.
\end{proof}

\section{Proof of Theorem \ref{thm}}
\label{s:pf}

We return to the situation from \eqref{eq:oplus} where $R$ is no longer assumed to be irreducible.
Recall that in view of Theorem \ref{thm},
we are interested in vectors 
\begin{eqnarray}
\label{eq:w}
w\in R'\setminus R \;\;\; \text{  with } \;\;\; w^2 = -1, -2.
\end{eqnarray}
Writing $w=(w_1,\hdots,w_m)$, we infer that each $w_i$ is either zero or a small vector in $R_i^\vee\setminus R_i$.
For the record, we note the following observation:

\begin{Lemma}
\label{lem:range}
Assume that $w_i\neq 0$. Then $w_i$ is  in one of the orbits listed in \ref{ss:1st} -- \ref{ss:4th}.
\end{Lemma}

\begin{Remark}
We will not need this in the sequel,
but the second smallest vectors in $A_n$ for $n\geq 8$
can also be excluded as follows:
we have
\[
-2 < (a_2^\vee)^2 = -\frac {2(n-1)}{n+1} < -\frac 32.
\]
Since $v^2\leq -\frac 12$ for any non-zero $v\in R^\vee$,
we infer that $a_2^\vee$ cannot be complemented by any vector from some dual of a root lattice to have square $-2$.
\end{Remark}

We are now in the position to attack the proof of Theorem \ref{thm}.
This puts us in the situation of \eqref{eq:w} 
with the crucial addition that the vertices $v_j$ of $R\subset\Num(S)$ are classes of $(-2)$-curves $C_j\subset S$.
Our argument is inspired by discussions with S.~Rams on the case of $R=4A_2$ on Enriques surfaces
(see \cite{RS}) which themselves received crucial input from \cite[Lem.~1.1]{br}.

Assume that $w\geq 0$ as a divisor on $S$.
We argue componentwise using the expression $w=(w_1,\hdots,w_m)$,
so let us fix the $i$th component and assume that $R_i^\vee\ni w_i\neq 0$,
so that $w_i$ is a small vector in the range of \ref{ss:1st} -- \ref{ss:4th} by Lemma \ref{lem:range}.
Numbering the vertices of $R_i$ by $v_1,\hdots, v_n$ as before,
there are two cases by Lemma \ref{lem:2cases}:
\begin{itemize}
\item
either $w_i.v_j\geq 0$ for each $j=1,\hdots,n$,
so $w_i=v^\vee$ for some simple vertex $v\in R_i$;
\item
or there is some vertex $v_j$ such $w_i.v_j=-1$.
\end{itemize}
We shall now show how the second case successively leads to the first.
Since $w_i.C_j=-1$ and $w\geq 0$, we infer that $C_j$ is contained in the support of $w_i$,
so $w_i-C_j$ is still effective.
On the other hand, incidentally, $w_i-C_j$ is the reflection of $w_i$ in $C_j$,
so the two vectors have the same square.
Thus we can iterate the above process.
Necessarily this procedure terminates since at each step,
the sum of the $\Q$-coefficients of the small vector,
expressed in the $v_j$ as in \eqref{eq:v}, drops by one
while there are only finitely many vectors of given square, of course.
In the end, we obtain a vector $w_i'\in R_i^\vee$ of the same square as $w_i$
(equivalent under reflections, in fact), such that
\[
w_i'.v_j\geq 0 \;\;\; \text{  for all } \;\; j=1,\hdots,n.
\]
As we have seen, the construction preserves effectivity: 
\begin{eqnarray}
\label{eq:>}
w_i'\geq 0.
\end{eqnarray}
On the other hand, $w_i'$ equals the dual vector of some simple vertex $v\in R_i$ by Lemma \ref{lem:2cases}.
But then all $\Q$ coefficients of $w_i'=v^\vee$ are negative by Lemma \ref{lem:a_i}.

Since the irreducible root lattices $R_i$ are orthogonal,
we can carry out the above procedure for all $w_i$ separately.
Ultimately, we arrive at an effective divisor $w'=(w_1',\hdots,w_m')$
all whose components have zero or negative $\Q$-coefficients in terms of
the given basis of $R$ consisting of $(-2)$-curves.
Hence there is some integral multiple $$Mw'\leq 0 \; \;(M\in\N),$$ giving the required contradiction
since $w'$ was still shown to be effective (as a consequence of \eqref{eq:>}), and clearly non-zero.

If $w\leq 0$, then reverse the sign and proceed as above.
This completes the proof of Theorem \ref{thm}.
\qed

\subsection*{Acknowledgements}

Many thanks to Slawomir Rams and Tetsuji Shioda for  helpful comments and discussions.


\begin{thebibliography}{99}

%\bibitem{A} Artin, M.: \emph{Supersingular $K3$ surfaces}, Ann.~scient.~\'Ec.~Norm.~Sup.~(4) {\bf 7} (1974), 543--568.
%

%\bibitem{ASD} Artin, M., Swinnerton-Dyer, P.: \emph{The Shafarevich-Tate conjecture for pencils of elliptic curves on $K3$ surfaces}, Invent.~Math.~{\bf 20} (1973), 249--266.
%%%
%%%
%%%\bibitem{badescu}  Badescu L.: \emph{Algebraic Surfaces}, Universitext, Springer-Verlag, New York, 2001. 
%%%
%%%
%%\bibitem{Barth}  Barth, W.~P.:   \emph{Lectures on K3- and Enriques surfaces.} in: {\it Algebraic geometry, Sitges (Barcelona 1983)}, Springer Lecture Notes in Math. {\bf 1124}  (1985), 21--57. 
%%%%
%%%%\bibitem{barth2000}   Barth, W.~P.:  \emph{$S_5$-symmetric quintic surfaces}, notes/personal communication, Erlangen, 2000.


\bibitem{b} Barth, W.: \emph{K3 Surfaces with Nine Cusps.}
Geom. Dedic. {\bf 72} (1998), 171--178.

%\bibitem{bpv} Barth, W., Hulek, K., Peters, C., van de Ven, A.:
%\emph{Compact complex surfaces}.
%Second edition,
%Erg.~der Math.~und ihrer Grenzgebiete,
%3.~Folge, Band {\bf 4}.~Springer (2004), Berlin.

\bibitem{br}
Barth, W.,
Rams, S.:
\emph{Equations of low-degree Projective Surfaces with three-divisible Sets of Cusps},
Math. Z. {\bf 249} (2005), 283--295.



%%
%%\bibitem{BHPV}
%%Barth, W., Hulek, K., Peters, C., van de Ven, A.:
%%\emph{Compact complex surfaces}.
%%Second edition,
%%Erg.~der Math.~und ihrer Grenzgebiete,
%%3.~Folge, Band {\bf 4}.~Springer (2004), Berlin.
%%
%%\bibitem{BP}
%%Barth, W.,
%%Peters, C,:
%%\emph{Automorphisms of Enriques Surfaces},
%%Invent, Math. {\bf 73} (1983), 383--412.
%
%
%\bibitem{Beauville} Beauville, A.:
%\emph{Les familles stables de courbes elliptiques sur $\PP^1$
%admettant quatre fibres singuli\'eres},
%C.R.~Acad.~Sci.~Paris {\bf 294} (1982), 657--660.
%
%
%\bibitem{BC}
%Blanc, J.,
%Cantat, S.:
%\emph{Dynamical degrees of birational transformations of projective surfaces},
%preprint (2013),
%arXiv: 1307.0361.

%
%
%\bibitem{BM1} Bombieri, E., Mumford, D.:
%\emph{Enriques' classification of surfaces in char.~$p$. II},
%in: Baily, W.~L.~Jr., Shioda, T.~(eds.),
% \emph{Complex analysis and algebraic geometry},
% Iwanami Shoten, Tokyo (1977) 23--42.
%
%
%\bibitem{BM2} Bombieri, E., Mumford, D.:
%\emph{Enriques' classification of surfaces in char.~$p$. III},
%Invent.~Math.~{\bf 35} (1976), 197--232.




%
%\bibitem{BS}
%Boissiere, S., Sarti, A.:
%\emph{Counting lines on surfaces},
%Ann. Scuola Norm. Sup. Pisa Cl. Sci. {\bf 5} (2007), 39--52.
%


\bibitem{bourbaki-lie}  Bourbaki, N.: \emph{\'El\'ements de math\'ematique. Fasc. XXXIV. Groupes et alg\'ebres de Lie. 
Chapitre IV -- VI.}
%Chapitre IV: Groupes de Coxeter et syst\'emes de Tits. 
%Chapitre V: Groupes engendr\'es par des r\'eflexions. Chapitre VI: syst\'emes de racines.}  
Actualit\'es Scientifiques et Industrielles, No. {\bf 1337}. Hermann (1968), Paris.


%\bibitem{Cayley}
%Cayley, A.:
%\emph{On the triple tangent planes of surfaces of the third order}, 
%Cambridge and Dublin Math.~J.~{\bf 4} (1849), 118--138.
%
%\bibitem{CS}
%Cartwright, D., Steger, T.:
%\emph{Enumeration of the 50 fake projective planes}, 
%Comptes Rendus Math. {\bf 348} (2010), 11--13.

%
%\bibitem{Clebsch}
%Clebsch, A.:
%\emph{Zur Theorie der algebraischen Fl\"achen},
%Journal reine angew. Math. {\bf 58} (1861), 93--108. 

\bibitem{CS} 
Conway, J., 
Sloane, N.: 
Sphere Packings, Lattices and
Groups, 3rd edition, Springer-Verlag, New York (1999).



%
%
%\bibitem{Cox} Cox, D.~A.: \emph{Primes of the Form $x^2+ny^2:$ Fermat, Class Field Theory, and  
%Complex Multiplication}, Wiley Interscience
%(1989).
%
%\bibitem{Deligne}
%Deligne, P.:
%\emph{Relevement de surfaces K3 en charact\'eristique nulle},
%Lect.~Notes Math.~{\bf 868} (1981), 58--71.
%
%
%\bibitem{DL}
%Dolgachev, I.,
%Liedtke, C.:
%\emph{Enriques surfaces II},
%manuscript (2017).
%
%%
%%\bibitem{DK}
%%Dolgachev, I.,
%%Keum, J.H.:
%%\emph{K3 surfaces with a symplectic automorphism of order 11},
%%JEMS {\bf 11} (2009), 799--818.
%
%\bibitem{EHSB}
%Ekedahl, T.,
%Hyland, J. M. E.,
%Shepherd-Barron, N. I.:
%\emph{Moduli and periods of simply connected Enriques surfaces},
%preprint (2012),
%arXiv: 1210.0342.
%
%
%%
%\bibitem{ES-2}
%Elkies, N. D., Sch\"utt, M.:
%\emph{Genus 1 fibrations on the supersingular K3 surface in characteristic 2 with Artin invariant 1},
%preprint (2012),
%arXiv: 1207.1239.
%
%
%\bibitem{EO}
%Esnault, H.,
%Oguiso, K.:
%\emph{Non-liftability of automorphism groups of a K3 surface in positive characteristic},
%preprint (2014),
%arXiv: 1406.2761.
%
%
%\bibitem{EOY}
%Esnault, H.,
%Oguiso, K.,
%Yu, X.:
%\emph{Automorphisms of elliptic K3 surfaces and Salem numbers of maximal degree},
%preprint (2014),
%arXiv: 1411.0769.
%
%
%\bibitem{ES}
%Esnault, H., 
%Srinivas, V.: 
%\emph{Algebraic versus topological entropy for surfaces over finite fields}, 
%Osaka J. Math. {\bf 50} (2013), 827--846.
%
%%
%
%%\bibitem{HT}
%%Harris, J.,
%%Tschinkel, Y.:
%%\emph{Rational points on quartics},
%%Duke Math.~J.~{\bf 104} (2000),  477--500. 
%
%
%\bibitem{HS}  Hulek, K., Sch\"utt, M.: \emph{Enriques surfaces and Jacobian elliptic surfaces},
%Math. Z. {\bf 268} (2011), 1025--1056.
%
%
%
%\bibitem{HSa}
%Hulek, K., Sch\"utt, M.:
%\emph{Arithmetic of singular Enriques surfaces},
%Algebra \& Number Theory {\bf 6} (2012), 195--230.
%
%
%
%\bibitem{HK}
%Hwang, D., Keum, JH:
%\emph{The maximum number of singular points on rational homology projective planes}, J. Algebraic
%Geom. 
%{\bf 20} (2010), 495--523.
%
%
%\bibitem{HKO}
%Hwang, D, Keum, JH, Ohashi, H.:
%\emph{Gorenstein $\Q$-homology projective planes},
%Science China Math. {\bf 58} (2015), 501--512.
%
%\bibitem{Illusie}
%Illusie, L.:
%\emph{Complexe de de Rham-Witt et cohomologie cristalline}, Ann. Sci. \'Ecole Norm.
%Sup. {\bf 12} (1979), 501--661.
%
%
%\bibitem{Ito}
%Ito, H.:
%\emph{The Mordell--Weil groups of unirational quasi-elliptic surfaces in characteristic 2},
%Tohoku Math. J. {\bf 46} (1994), 221--251.
%
%%
%%%
%%%
%%%\bibitem{KS} Katsura, T., Shioda, T.: \emph{On Fermat varieties}, Tohoku Math.~J.~(2) {\bf 31} (1979), no.~1, 97--115.
%%%
%%
%\bibitem{KK}
%Katsura, T.,
%Kond\=o, S.:
%\emph{On Enriques surfaces in characteristic 2 with a finite group of automorphisms},
%to appear in J. Alg. Geom.,
%preprint (2015),
%arXiv:1512.06923.
%
%\bibitem{KKM}
%Katsura, T.,
%Kond\=o, S.,
%Martin, G.:
%\emph{Classification of Enriques surfaces with finite automorphism group in characteristic 2},
%preprint (2017),
%arXiv: 1703.09609.
%
%
%
%%
%
%\bibitem{Keum-Viet}
%Keum, JH:
%\emph{The moduli space of $\Q$-homology projective planes with 5 quotient singular points}, Acta Math. Vietnamica {\bf 35}
%(2010), 79--89.

%%
%%%\bibitem{RK} Kloosterman, R.: \emph{Elliptic K3 surfaces with geometric Mordell--Weil rank 15},  Canad.~Math.~Bull.~{\bf 50}  (2007),  no.~2, 215--226.
%%
%\bibitem{K} Kodaira, K.:
%\emph{On compact analytic surfaces I-III},
%Ann.~of Math., {\bf 71} (1960), 111--152;
%{\bf 77} (1963), 563--626; {\bf 78} (1963), 1--40.
%
%
%%
%\bibitem{Kondo-Enriques} Kond\=o, S.: \emph{Enriques surfaces with finite automorphism groups.}, Japan. J. Math. (N.S.) {\bf 12} (1986), 191--282. 
%%
%%
%%\bibitem{Kondo-Aut}
%%S.~Kond\=o, 
%%\emph{Automorphisms of algebraic K3 surfaces which act trivially
%%on Picard groups}, 
%%J.~Math.~Soc.~Japan, {\bf 44}, No.~1 (1992), 75--98.
%
%\bibitem{Lang1}
%Lang, W.~E.:
%\emph{Extremal rational elliptic surfaces in characteristic p. I. Beauville surfaces},
%Math. Z. {\bf 207} (1991), 429--437. 
%
%\bibitem{Lang2}
%Lang, W.~E.:
%\emph{Extremal rational elliptic surfaces in characteristic p. II. Surfaces with three or fewer singular fibres},
%Ark. Mat. {\bf 32} (1994),  423--448. 
%
%
%%%%
%%%%
%%%%% \bibitem{lang} Lang, W.~E.: \emph{Classical Godeaux Surface in Characteristic P,} Math. Ann.~{\bf 256} (1981), 419--427.
%%%%
%%%%
%
%\bibitem{LM}
%Lieblich, M.,
%Maulik, D.:
%\emph{A note on the cone conjecture for K3 surfaces in positive characteristic},
%preprint (2011),
%arXiv: 1102.3377v3.
%
%
%\bibitem{Liedtke}
%Liedtke, C.:
%\emph{Arithmetic Moduli and Lifting of Enriques Surfaces,}
%J. Reine Angew. Math. {\bf 706} (2015), 35--65. 
%
%
%\bibitem{Martin}
%Martin, G.:
%\emph{Enriques surfaces with finite automorphism group in positive characteristic},
%preprint (2017),
% arXiv: 1703.08419.
%
%
%\bibitem{MP} Miranda, R., Persson, U.:
%\emph{On Extremal Rational Elliptic Surfaces},
%Math.~ Z.~{\bf 193} (1986), 537--558.
%
%\bibitem{Mumford-surf}
%Mumford, D.:
%\emph{Enriques' classification of surfaces in char.~$p$. I},
%in: Global analysis.
%Princeton, University Press 1969.
%
%
%


%%
%%\bibitem{N}
%%Nikulin, V.V.:
%%\emph{Finite groups of automorphisms of K\"ahlerian K3 surfaces}, 
%%Trudy
%%Moskov. Mat. Obshch. {\bf 38} (1979), 75--137. English translation: Trans. Moscow Math.Soc. {\bf 38}
%%(1980), 71--135.
%%
%%%
%\bibitem{Nikulin} Nikulin, V.~V.:
%\emph{Integral symmetric bilinear forms and some of their applications},
%Math.~USSR Izv.~{\bf 14}, No.~1 (1980), 103--167.
%
%\bibitem{Nishi} Nishiyama, K.-I.: \emph{The Jacobian fibrations on some K3 surfaces and their Mordell--Weil groups}, Japan.~J.~Math.~{\bf 22} (1996), 293--347.
%
%
%\bibitem{OS} Oguiso, K., Shioda, T.: \emph{The Mordell--Weil lattice of a rational elliptic surface}, Comment.~Math.~Univ.~St.~Pauli~{\bf 40} (1991), 83--99.
%
%
%%
%\bibitem{Ogus} Ogus, A.:
%\emph{Supersingular K3 crystals},
%Journ\'ees de G\'eom\'etrie Alg\'ebrique de Rennes (Rennes 1978),
%Vol.~II,  3--86, Ast\'erisque {\bf 64}, Soc.~Math.~France, Paris, 1979.

%
%\bibitem{PSS} Piatetski-Shapiro, I.~I., Shafarevich, I.~R.: \emph{Torelli's theorem for algebraic surfaces of type ${\rm K}3$}, Izv.~Akad.~Nauk SSSR Ser.~Mat.~{\bf 35} (1971), 530--572.
%
%\bibitem{PY}
%Prasad, G., Yeung, S.-K.:
%\emph{Fake projective planes}, 
%Invent. Math. {\bf 168} (2007), 321--370;
%addendum:  Invent. Math. {\bf 182} (2010), 213--227.
%
%%
%%%\bibitem{reid}
%%%Reid, M.:
%%%\emph{Campedelli versus Godeaux}. 
%%%Problems in the theory of surfaces and their classification (Cortona, 1988), 
%%%Sympos. Math. {\bf XXXII}, 
%%%309--365.
%%%

\bibitem{RS}
Rams, S., Sch\"utt, M.:
\emph{On Enriques surfaces with four cusps},
preprint (2015), arXiv: 1404.3924v2.

%
%\bibitem{RuS}
%Rudakov, A.~N., Shafarevich, I.~R.:
%\emph{Surfaces of type K3 over fields of finite characteristic}
%J.~Sov.~Math.~{\bf 22} (1983), no.~4, 1476--1533.
%


%%%
%%%\bibitem{S-NS} Sch\"utt, M.: \emph{K3 surfaces of Picard rank 20 over $\Q$}, Algebra \& Number Theory {\bf 4} (2010), no.~3, 335--356.
%%%
%%%%
%%%%\bibitem[S10]{S-dmv} Sch\"utt, M.: \emph{Arithmetic of K3 surfaces}, to appear in Jahresbericht der DMV, preprint (2008), arXiv: 0808.1061.
%%%%
%%%\bibitem{S-quintic} Sch\"utt, M.: \emph{Quintic surfaces with maximum and other Picard numbers}, 
%%%Journal Math. Soc. Japan {\bf 63} (2011), 1187--1201. 
%%%
%
%%%\bibitem{SS} Sch\"utt, M., Shioda, T.:
%%%\emph{An interesting elliptic surface over an elliptic curve},
%%%Proc.~Jap.~Acad.~{\bf 83}, 3 (2007), 40--45. 
%
%
%\bibitem{S-Fields} Sch\"utt, M.:
% \emph{Fields of definition of singular K3 surfaces},
% Communications in Number Theory and Physics {\bf 1}, 2 (2007), 307--321.
%

\bibitem{S-Q-hom} Sch\"utt, M.:
 \emph{Moduli of Gorenstein $\Q$-homology projective planes},
preprint (2016),
arXiv: 1505.04163v2.

\bibitem{S-Q_l}
Sch\"utt, M.:
\emph{$\Q_\ell$-cohomology projective planes from Enriques surfaces in odd characteristic},
preprint (2016),
arXiv: 1611.03847.

\bibitem{S-Q_2}
Sch\"utt, M.:
\emph{$\Q_\ell$-cohomology projective planes and Enriques surfaces in  characteristic two},
preprint (2017),
arXiv: 1703.10441.


%
%\bibitem{S-dyn}
%Sch\"utt, M.:
%\emph{Dynamics on supersingular K3 surfaces}
%Comment.~Math.~Helv.~{\bf 91} (2016), 705--719.




%
%%
%%

%\bibitem{RS}
%Rams, S., Sch\"utt, M.:
%\emph{64 lines on smooth quartic surfaces},
%preprint (2013), arXiv: 1212.3511v3.
%
%
%%\bibitem{RS} Rams, S., Sch\"utt, M.: \emph{On quartics with lines of the second kind}, 
%%preprint (2013), to appear in Adv. Geom., arXiv: 1303.1304.
%
%\bibitem{Salmon}
%Salmon, G.:
%\emph{On the triple tangent planes to a surface of the third order,} 
%Cambridge and Dublin Math.~J.~{\bf 4} (1849), 252--260.
%
%%
%%\bibitem{schur} Schur, F.: \emph{Ueber eine besondre Classe von Fl\"achen vierter Ordnung.} 
%%Math. Ann. {\bf 20} (1882), 254--296.
%
%\bibitem{S-sigma=1}
%Sch\"utt, M.:
%\emph{A note on the supersingular K3 surface of Artin invariant 1},
%Journal of Pure and Applied Algebra {\bf 216} (2012), 1438--1441. 
%
%\bibitem{SS}
%Sch\"utt, M., 
%Schweizer, A.:
%\emph{On the uniqueness of K3 surfaces with maximal singular fibre}, 
%Annales de l'institut Fourier {\bf 63} (2013), 689--713. 
%
%\bibitem{SSh} 
% Sch\"utt, M., Shioda, T.:
% \emph{Elliptic surfaces},
%Algebraic geometry in East Asia - Seoul 2008, 
%Advanced Studies in Pure Math.~{\bf 60} (2010), 51-160.  
%
%%\bibitem{SSvL} Sch\"utt, M., Shioda, T., van Luijk, R.: \emph{Lines on Fermat surfaces}, 
%%J.~Number Theory {\bf 130} (2010), 1939--1963.
%
%%%
%%%
%%%
%%%\bibitem{vGS}
%%%Sch\"utt, M.,
%%%van Geemen, B.:
%%%\emph{Two moduli spaces of abelian fourfolds with an automorphism of order five},
%%%to appear in Int.~J.~Math.,
%%% arXiv: 1010.3897.
%% 
%
%%\bibitem{Segre}
%%Segre, B.:
%%\emph{The maximum number of lines lying on a quartic surface}, 
%%Quart. J. Math., Oxford
%%Ser. {\bf 14} (1943), 86--96.
%%
%%\bibitem{Segre2}
%%Segre, B:
%%\emph{Forme e geometrie hermitiane, con particolare riguardo al caso finito},
%%Ann.~Mat.~Pura Appl.~(4) {\bf 70} (1965), 1--201.
%%
%%\bibitem{Sengupta}
%%Sengupta, T.:
%%\emph{Elliptic Fibrations on Supersingular K3 Surface with Artin Invariant 1 in characteristic 3},
%%preprint (2012), 
%%arXiv: 1204.6478.
%%
%
%
% \bibitem{Shimada-rk21} 
% Shimada, I.:  
%\emph{Rational double points on supersingular $K3$ surface}, 
%Mathematic of computation {\bf 73}, No. 248 (2004), 1989--2017.
%
%
% \bibitem{Shimada} Shimada, I.:  
%\emph{Transcendental lattices and supersingular reduction lattices of a singular $K3$ surface}, 
%Trans AMS {\bf 361} (2009), 909--949.

%
%
%\bibitem{ShEMS} Shioda, T.:
% \emph{On elliptic modular surfaces},
% J.~Math.~Soc.~Japan {\bf 24} (1972), 20--59.
%
%
%%\bibitem{Sh-int}
%%Shioda, T.:
%%\emph{An example of unirational surfaces in characteristic p},
%%Math.~Ann.~{\bf 211} (1974), 233--236. 
%
%%
%%\bibitem{Sh}
%%Shioda, T.:
%%\emph{An explicit algorithm for computing the Picard number of certain algebraic surfaces},
%%Amer.~J.~Math.~{\bf 108}, No. 2 (1986), 415--432.
%% 
%%
%%%
%\bibitem{ShMW} Shioda, T.: \emph{On the Mordell--Weil lattices}, Comm.~Math.~Univ.~St.~Pauli {\bf 39} (1990), 211--240.
%
%\bibitem{Sh-Murre}
%Shioda, T.:
%\emph{Correspondence of elliptic curves and Mordell--Weil lattices of certain elliptic K3 surfaces},
%in \emph{Algebraic Cycles and Motives}, Vol.~2, 
%LMS Lect.~Note Ser.~{\bf 344}, Cambridge Univ. Press (2007),  319--339.
%
%%
%%% \bibitem{shioda}
%%%Shioda, T.:
%%%{\it The Mordell--Weil lattice of $y^2=x^3 + t^5 - 1/t^5 -11$}, 
%%%Comment. Math. Univ. St. Pauli {\bf 56} (2007), 45--70.
%
%
%\bibitem{SI} Shioda, T., Inose, H.: \emph{On Singular $K3$ Surfaces}, in: W. L. Baily Jr., T. Shioda (eds.), \emph{Complex analysis and algebraic geometry}, Iwanami Shoten, Tokyo (1977), 119--136.
%%
%\bibitem{SM} Shioda, T., Mitani, N.:
% \emph{Singular abelian surfaces and binary quadratic forms},
% in: \emph{Classification of algebraic varieties
% and compact complex manifolds},
% Lect.~Notes in Math.~{\bf 412} (1974), 259--287.
% 
%
%
%%%
%%
%%
%\bibitem{Tate} Tate, J.: {\it Algorithm for determining the type
%of a singular fibre in an elliptic pencil}, in: {\it Modular
%functions of one variable IV} (Antwerpen 1972), Lect.~Notes in Math.~{\bf 476}
%(1975), 33--52.
%
%\bibitem{Tate-C} Tate, J.:
% {\it Algebraic cycles and poles of zeta functions},
%in: {\it Arithmetical Algebraic Geometry}
%(Proc. Conf. Purdue Univ., 1963), 93--110, Harper \& Row (1965).
%
%%\bibitem{T1} Tate, J.: \emph{On the conjectures of Birch and Swinnerton-Dxer and a geometric analog}, in: A. Grothendieck, N. H. Kuiper (Hrsg.), \emph{Dix expos\'es sur la cohomologie des schemas}, 189-214. North-Holland Publ., Amsterdam, 1968.

%%%\bibitem{T2} Tate, J.: \emph{Conjectures on algebraic cycles in $\ell$-adic cohomology}, in: Proc.~Sympos.~Pure Math., Vol.~{\bf 55}, 71-83. Providence, RI: Amer.~Math.~Soc., 1994.
%%
%% \bibitem{vL} van Luijk, R.: \emph{K3 surfaces with Picard number one and infinitely many rational points},  Algebra \& Number Theory  {\bf 1} (2007),  no.~1, 1--15.

%\bibitem{Voloch}
%Voloch, F.:
%\emph{Surfaces in $\PP^3$ over finite fields},
%Topics in Algebraic and Noncommutative Geometry,
%Contemp.~Math.~{\bf 324} (2003), 219--226.
%
%%
%
%\bibitem{xiao}
%Xiao, G.:
%\emph{Galois covers between K3 surfaces},
%Annales de l'institut Fourier {\bf 46} (1996), 73--88. 

%%\bibitem{xie2010} Xie, Jinjing: \emph{More Quintic Surfaces with 75 Lines},
%%Rocky Mountain J. Math. {\bf 40} (2010), 2063--2089.
%
%\bibitem{Ye}
% Ye, Q.:
% \emph{On Gorenstein log del Pezzo surfaces}, 
% Japan J. Math. {\bf 28} (2002), 87--136.

%%

\end{thebibliography}
\end{document}